\documentclass[a4,12pt]{amsart}
\usepackage{amssymb,amstext,amsmath,amscd,amsfonts,enumerate}
\usepackage{hyperref}
\usepackage[all]{xy}
\usepackage{amsthm}
\usepackage{enumerate}
\usepackage{diagbox}
\usepackage{color}
\usepackage{setspace}
\usepackage{multirow,multicol}
\usepackage{tikz}
\usetikzlibrary{shapes.geometric, arrows}
\usetikzlibrary{calc}
\newtheorem{theorem}{Theorem}[section]
\newtheorem{theorem*}{Theorem}
\newtheorem{corollary}[theorem]{Corollary}
\newtheorem{corollary*}[theorem*]{Corollary}
\newtheorem{lemma}[theorem]{Lemma}
\newtheorem{proposition}[theorem]{Proposition}

\theoremstyle{definition}
\newtheorem{definition}[theorem]{Definition}

\newtheorem{question}[theorem]{Question}
\newtheorem*{question*}{Question}

\newtheorem*{conjecture*}{Conjecture}
\newtheorem{example}[theorem]{Example}

\newtheorem*{notation*}{Notation}

\newtheorem*{claim*}{Claim}

\def\ssum{\textstyle\sum\limits}
\def\D{\mathsf{D}}
\def\k{\mathsf{k}}
\def\H{\mathcal{H}}
\def\I{\mathcal{I}}
\def\S{\mathfrak{S}}
\def\End{\operatorname{End}}
\def\Hom{\operatorname{Hom}}
\def\dim{\operatorname{dim}}
\def\mod{\operatorname{\mathsf{mod}\text{-}}\hspace{-0.03in}}
\def\proj{\operatorname{\mathsf{proj}\text{-}}\hspace{-0.03in}}
\def\Kb{\mathsf{K}^{\rm b}}

\def\soc{\mathsf{soc}\hspace{0.03in}}
\newcommand{\stautilt}{\mbox{\sf s$\tau$-tilt}\hspace{0.03in}}

\newcommand{\ctext}[1]{\raise0.2ex\hbox{\textcircled{\scriptsize{#1}}}}

\begin{document}
\setlength{\baselineskip}{16pt}

\title
[On \(\tau\)-tilting finiteness of symmetric algebras]
{On $\tau$-tilting finiteness of symmetric algebras of polynomial growth}

\author
[K. Miyamoto and Q. Wang]
{Kengo Miyamoto and Qi Wang}

\address{Department of Computer and Information Science, Graduate School of Science and Engineering, Ibaraki University, Hitachi, Ibaraki 316-8511, Japan.}
\email{kengo.miyamoto.uz63@vc.ibaraki.ac.jp}

\address{Yau Mathematical Sciences Center, Tsinghua University, Beijing 100084, China.}
\email{infinite-wang@outlook.com}

\date{\today}

\begin{abstract}
In this paper, we report on the $\tau$-tilting finiteness of some classes of finite-dimensional algebras over an algebraically closed field, including symmetric algebras of polynomial growth, $0$-Hecke algebras and $0$-Schur algebras. Consequently, we find that derived equivalence preserves the $\tau$-tilting finiteness over symmetric algebras of polynomial growth, and self-injective cellular algebras of polynomial growth are $\tau$-tilting finite. Furthermore, the representation-finiteness and $\tau$-tilting finiteness over $0$-Hecke algebras and $0$-Schur algebras (with few exceptions) coincide.
\end{abstract}

\maketitle


\section{Introduction}
Throughout this paper, we will use the symbol $\k$ to denote an algebraically closed field. An algebra is assumed to be an associative basic connected finite-dimensional $\k$-algebra, unless otherwise specified. Modules are always finitely generated and right. For an algebra $A$, we denote by $\mod A$ (resp. $\proj A$) the category of modules (resp. projective modules) over $A$. 
We also denote by $\Kb(\proj A)$ the perfect derived category of $A$.
Finally, we use the symbols $\D$ and $\tau$ for the standard $\k$-dual and the Auslander-Reiten translation, respectively.

Recently, Adachi, Iyama and Reiten proposed $\tau$-tilting theory in \cite{AIR} to classify torsion classes in $\mod A$. They introduced a new class of $A$-modules called \emph{support $\tau$-tilting modules} (see Section \ref{sect2} for the definition), which are in bijective correspondence with several sets of important objects arising in representation theory, such as two-term silting complexes in $\Kb(\proj A)$ (\cite{AIR}), functorially finite torsion classes in $\mod A$ (\cite{AIR}), left finite semibricks in $\mod A$ (\cite{Asai1}), wide subcategories of $\mod A$ (\cite{MSt}), some $t$-structures and co-$t$-structures in $\Kb(\proj A)$ (\cite{KY}), and so on. In this context, those algebras admitting only finitely many support $\tau$-tilting modules, which are said to be \emph{$\tau$-tilting finite}, have been actively researched, for example, see \cite{Ada2, AAC, Aihara-Honma, AHMW, MW, W, Z}. Demonet, Iyama, and Jasso originally studied $\tau$-tilting finite algebras and characterized such algebras in \cite{DIJ}.

It is known that $\tau$-tilting finiteness is preserved under Morita equivalence, but usually not preserved under derived equivalence. Examples are easy to find: a path algebra $A$ of a Euclidean quiver is $\tau$-tilting infinite; taking a tilting $A$-module $M$ having both a preprojective and a preinjective direct summand, then the endomorphism algebra $B=\End_A M$ is representation-finite (see \cite[VIII, Proposition 4.4]{ASS}) and hence, $\tau$-tilting finite.

The notion of symmetric algebras plays a prominent role in representation theory. An algebra $A$ is said to be \emph{symmetric} if there exists a non-degenerate symmetric $\k$-bilinear form $(-, -): A\times A\rightarrow \k$ such that $(ab, c)=(a, bc)$, for all $a,b,c \in A$. 
Classical examples of symmetric algebras include group algebras of finite groups, Brauer graph algebras, and trivial extension algebras. Here, the \emph{trivial extension} $\mathsf{Triv}(A)$ of an algebra $A$ is defined by $\mathsf{Triv}(A):=A\ltimes \D(A)$ with multiplication $(a,f)(b,g)=(ab, ag+fb)$ for $a, b \in A$, $f, g \in\D(A)$. One may easily find that each algebra $A$ is a certain quotient of its trivial extension $\mathsf{Triv}(A)$, indicating the significance of symmetric algebras in representation theory.

One may expect thorough research on the behavior of symmetric algebras in $\tau$-tilting theory. There are outcomes for some subclasses of symmetric algebras: Aihara shows in \cite{Aihara-symmetric-alg} that any representation-finite symmetric algebra is tilting-connected; Adachi, Aihara and Chan provide in \cite{AAC} a combinatorial classification of two-term tilting complexes for Brauer graph algebras; Asashiba, Mizuno and Nakashima in \cite{AMN} along with Aoki in \cite{Aoki-Brauer-tree} use different approaches to determine the number of two-term tilting complexes for Brauer tree algebras; Koshio and Kozakai classify in \cite{Koshio-Kozakai} support $\tau$-tilting modules for certain block algebras of finite groups; Li and Zhang study in \cite{LZ-trivial-extension} support $\tau$-tilting modules over trivial extension algebras; Adachi and Aoki compute in \cite{AA} the number of two-term silting complexes for symmetric algebras with radical cube zero; to name a few.

We have learned from \cite{AAC} that the $\tau$-tilting finiteness of Brauer graph algebras is preserved under derived equivalence. From this result, it is reasonable to consider the following question:
\begin{question}\label{ques1}
Does derived equivalence preserve the $\tau$-tilting finiteness over symmetric algebras?
\end{question}

In this paper, the main strategy addressing Question \ref{ques1} is given in terms of the representation type of symmetric algebras. Due to Drozd's Tame-Wild Dichotomy \cite{Drozd}, a finite-dimensional algebra over an algebraically closed field is either tame or wild. We then consider the $\tau$-tilting finiteness of symmetric algebras in different types outlined in the following hierarchy (see \cite[(1.5)]{Sk06} for more details).
\begin{center}
\scalebox{0.8}{
\begin{tikzpicture}
\draw (-3,-1.1) rectangle (6,2.3);
\draw (-2,-0.3) rectangle (5,2.2);
\draw (-1,0.5) rectangle (4,2.1);
\draw (0,1.3) rectangle (3,2) ;
\draw (1.5,1.7) rectangle (1.5,1.7) node{Finite};
\draw (1.5,0.9) rectangle (1.5,0.9) node{Domestic};
\draw (1.5,0) rectangle (1.5,0) node{Polynomial growth};
\draw (1.5,-0.8) rectangle (1.5,-0.8) node{Tame};
\end{tikzpicture}}
\end{center}
It is well-known that each of the inclusions above is proper.
We mention a crucial property from \cite{Ric89(2)} that the tameness and the hierarchy above are preserved by the derived equivalence of symmetric algebras (or more generally, self-injective algebras, see \cite{Asashiba}). 

In the case of representation-finite symmetric algebras, a positive answer to Question \ref{ques1} is straightforward, since any representation-finite algebra is $\tau$-tilting finite. As the first main result in this paper, we provide an affirmative answer to Question \ref{ques1} for representation-infinite symmetric algebras of polynomial growth, as follows.

\begin{theorem*}[{= Theorem \ref{tau-finite_PG}}]\label{main-result-1}
Let $A$ and $B$ be two representation-infinite symmetric algebras of polynomial growth. If $A$ is derived equivalent to $B$, then the following conditions are equivalent.
\begin{enumerate}
\item[\rm{(1)}] $A$ is $\tau$-tilting finite.
\item[\rm{(2)}] $B$ is $\tau$-tilting finite.
\item[\rm{(3)}] The Cartan matrix $C_A$ (or equivalently, $C_B$) is non-singular.
\end{enumerate}
\end{theorem*}

It is noteworthy that Question \ref{ques1} remains unanswered for tame symmetric algebras of non-polynomial growth and wild symmetric algebras.

As an application of Theorem \ref{main-result-1}, one can determine the $\tau$-tilting finiteness of self-injective cellular algebras of polynomial growth, which are classified in \cite{AKMW} up to Morita equivalence (under the assumption that the characteristic of $\k$ is not $2$).

\begin{corollary*}[{= Corollary \ref{finite_cellular}}]
Suppose the characteristic of $\k$ is not $2$. Then, any algebra which is Morita equivalent to a self-injective cellular algebra of polynomial growth is $\tau$-tilting finite.
\end{corollary*}

Another motivation of this paper is to determine the $\tau$-tilting finiteness of $0$-Hecke algebras and $0$-Schur algebras, see Section \ref{section-4} for precise definitions. This may be independent of Question \ref{ques1}, but this is meaningful for the research related to Hecke algebras and $q$-Schur algebras. 
In fact, the $\tau$-tilting finiteness of Hecke algebras of type $\mathbb{A}$ and classical Schur algebras has already been determined in  \cite{ALS-Hecke-alg, AW, W2}.

Let $q$ be an indeterminate element. For a finite Coxeter system $(W,S)$, the \emph{Iwahori-Hecke algebra} $\H_{\k,q}(W)$ is the $\k$-algebra generated by $\{T_{s} \mid s\in S\}$ with quadratic relations and braid relations. Then,  $\H_0(W):=\H_{\k,0}(W)$ is called the \emph{$0$-Hecke algebra}. The structure and the representation theory of $\H_0(W)$ have been extensively studied, for example, see \cite{carter-0-hecke-alg, Fayers-0-hecke-alg, Norton-0-hecke-alg}. As the second main result in this paper, we classify $0$-Hecke algebras in terms of $\tau$-tilting finiteness as follows.

\begin{theorem*}[{= Theorem \ref{result-0-Hecke}}]\label{main3}
Let $W$ be an irreducible finite Coxeter group. Then, the following statements are equivalent.
\begin{enumerate}
\item[\rm{(1)}] The $0$-Hecke algebra $\H_0(W)$ is $\tau$-tilting finite.
\item[\rm{(2)}] $W$ is one of types $\mathbb{A}_1$, $\mathbb{A}_2$, $\mathbb{B}_2$ and $\mathbb{I}_2(m)$, for any $m\geq 5$.
\end{enumerate}
\end{theorem*}

We denote by $\S_r$ the symmetric group of degree $r$, which is actually a Coxeter group of type $\mathbb{A}_{r-1}$. 
According to \cite[(4.3.3)]{DY-0-Schur-alg}, the \emph{$0$-Schur algebra} $S_0(n,r)$ is Morita equivalent to a certain idempotent truncation of the $0$-Hecke algebra $\H_0(\S_r)$, for any $n, r\geq 2$. 
This class of algebras has been studied by various authors, such as Donkin \cite{dokin-q-schur-alg}, Su \cite{Su-0-schur-alg},  Deng and Yang \cite{DY-0-Schur-alg}, etc. As the third main result in this paper, we determine the $\tau$-tilting finiteness of $0$-Schur algebras as follows.

\begin{theorem*}[{= Theorem \ref{main-0-schur}}]\label{main4}
The following assertions hold.
\begin{enumerate}
\item[\rm{(1)}] For $n\geq 3$, the $0$-Schur algebra $S_0(n,r)$ is $\tau$-tilting finite if and only if $r=2,3$.
\item[\rm{(2)}] For $n=2$ and $r\geq 2$, the $0$-Schur algebra $S_0(2,r)$ is $\tau$-tilting finite.
\end{enumerate}
\end{theorem*}

\section{Preliminaries}\label{sect2}
Any finite-dimensional $\k$-algebra $A$ can be presented as a bound quiver algebra $\k Q/\I$, where $Q$ is a finite (connected) quiver and $\I$ is an admissible ideal of the path algebra $\k Q$. A quiver is said to be \emph{acyclic} if it admits no oriented cycle. We will commonly write $a \xrightarrow{\alpha} b$ to indicate that an arrow $\alpha$ has source $a$ and target $b$. We refer to some textbooks, e.g., \cite{ASS}, for more background materials on the representation theory of quivers. 

We recall some crucial properties of symmetric algebras. It is well-known that any algebra that is derived equivalent to a symmetric algebra is also symmetric, see \cite[Corollary 5.3]{Ric91}.
Moreover, the following fact is the reason why we can consider symmetric algebras in different hierarchies of representation types.
\begin{proposition}[{\cite[Corollary 2.2]{Ric89(2)} and \cite[Corollary 2]{KZ}}]\label{sym_hiera}
Let $A$ be a symmetric algebra and $B$ an algebra that is derived equivalent to $A$. If $A$ is tame (respectively, representation-finite, domestic, of polynomial growth), then $B$ has the same property.
\end{proposition}

We then review the basic definitions in $\tau$-tilting theory and collect some results on $\tau$-tilting finite algebras, which are needed for this paper. We refer to \cite{AIR, DIJ, DIRRT, EJR-reduction} for more aspects related to $\tau$-tilting theory.

\begin{definition}
Let $M\in \mod A$ and $|M|$ be the number of isomorphism classes of indecomposable direct summands of $M$.
\begin{enumerate}
\item[\rm{(1)}] $M$ is \emph{$\tau$-rigid} if $\Hom_A(M,\tau M)=0$.
\item[\rm{(2)}] $M$ is \emph{$\tau$-tilting} if it is $\tau$-rigid with $|M|=|A|$.
\item[\rm{(3)}] $M$ is \emph{support $\tau$-tilting} if there is an idempotent $e\in A$ such that $M$ is a $\tau$-tilting module over $A/AeA$.
\end{enumerate}
\end{definition}

We denote by $\stautilt A$ the set of isomorphism classes of basic support $\tau$-tilting $A$-modules; it admits a poset structure, see \cite[Subsection 2.4]{AIR} for details.
We recall that $A$ is called \emph{$\tau$-tilting finite} if $\stautilt A$ is a finite set, and otherwise, $A$ is said to be \emph{$\tau$-tilting infinite}.
According to \cite{DIJ}, an algebra $A$ is $\tau$-tilting finite if and only if $A$ has only finitely many isomorphism classes of bricks.

\begin{example}\label{ex1} 
We give some examples of $\tau$-tilting finite/infinite algebras.
\begin{enumerate}
\item[\rm(1)] A local algebra $A$ has exactly two basic support $\tau$-tilting modules $A$ and $0$. Thus, any local algebra is $\tau$-tilting finite.

\item[\rm(2)] A representation-finite algebra is $\tau$-tilting finite.

\item[\rm(3)] Any module lying on a preprojective (or preinjective) component is a brick. Thus, representation-infinite algebras with preprojective components are $\tau$-tilting infinite (see \cite[Remark 2.8]{Mo}). Moreover, the path algebra $\k Q$ of an acyclic quiver $Q$ is $\tau$-tilting finite if and only if $Q$ is a Dynkin quiver (see \cite[Theorem 2.6]{Ada2}).
\end{enumerate}
\end{example}

Let $\I$ be a two-sided ideal of $A$ and $e$ an idempotent of $A$. By considering the following fully faithful functors
\[ 
-\otimes_{A/\I} A/\I: \mod (A/\I) \to\mod A, \quad \Hom_{eAe}(Ae, - ):\mod eAe \to \mod A, 
\]
we have the fact that $\tau$-tilting finiteness is preserved under taking quotient (or factor) and idempotent truncation, see \cite{DIJ, DIRRT, Pl19}.

The \emph{Cartan matrix} $C_A$ of an algebra $A=\k Q/\I$ is the matrix whose $(i,j)$-entry is $\dim_{\k}\Hom_A(P_i,P_j)$, where $P_i$ is the indecomposable projective $A$-module associated with the vertex $i$ in $Q$. If the determinant of $C_A$ is equal to zero, it is referred to as a \emph{singular matrix}; otherwise, it is termed as a \emph{non-singular matrix}. If $A$ is symmetric, then $C_A$ is a symmetric matrix and furthermore, $C_A$ is said to be \emph{positive definite} if all eigenvalues of $C_A$ are positive. The following statement is a well-known fact, e.g., see \cite{Ric89}.

\begin{proposition}\label{cartan-derived}
Let $A$ be a symmetric algebra and $B$ an algebra which is derived equivalent to $A$. If the Cartan matrix $C_A$ is positive definite, then so is $C_B$.
\end{proposition}
\begin{proof}
For an algebra $\Lambda$, we denote by $\mathsf{K_0}(\proj \Lambda)$ and $\mathsf{K_0}(\Kb(\proj \Lambda))$ the Grothendieck groups of $\proj \Lambda$ and $\Kb(\proj \Lambda)$, respectively. Then it is well-known that the canonical embedding $\proj \Lambda \to \Kb(\proj \Lambda)$ induces an isomorphism $\mathsf{K_0}(\proj \Lambda) \xrightarrow{\ \sim\ }\mathsf{K_0}(\Kb(\proj \Lambda))$ (e.g., see \cite{Gr72}).
For a complex $M$ in $\Kb(\proj \Lambda)$, we write $[M]$ for the equivalence class of $M$ in the Grothendieck group $\mathsf{K_0}(\Kb(\proj \Lambda))$. Let $\langle -,- \rangle_\Lambda: \mathsf{K_0}(\Kb(\proj \Lambda))\to \mathbb{Z}$ be the bilinear form defined by
\[ 
\langle [M],[N] \rangle_\Lambda= \ssum_{\ell=0}^\infty (-1)^{\ell}\dim_\k \Hom_{\Lambda}([M],\Sigma^{\ell}[N]), 
\]
where $\Sigma$ is the shift functor on $\Kb(\proj \Lambda)$.
It determines a bilinear form $\langle -,- \rangle_\Lambda: \mathsf{K_0}(\proj \Lambda)\to \mathbb{Z}$, which is represented by $C_\Lambda$ with respect to the basis obtained by the isomorphism classes of indecomposable projective $\Lambda$-modules.

Let $A$ be a symmetric algebra. If $A$ and $B$ are derived equivalent, then it follows from \cite[Theorem 6.4]{Ric89} that there exists a triangle equivalence $F:\Kb(\proj A)\to \Kb(\proj B)$. The triangle equivalence induces an isomorphism $\mathsf{K_0}(\Kb(\proj A))\to \mathsf{K_0}(\Kb(\proj B))$ which preserves the bilinear forms.
Indeed, for $[M], [N]\in \mathsf{K_0}(\Kb(\proj A))$, we have
\begin{align*}
\langle F[M], F[N] \rangle_B & = \ssum_{\ell=0}^\infty(-1)^{\ell}\dim_\k \Hom_{B}(F[M],\Sigma^{\ell}(F[N])) \\
& = \ssum_{\ell=0}^\infty(-1)^{\ell}\dim_\k \Hom_{B}(F[M],F(\Sigma^{\ell}[N])) \\
 & = \ssum_{\ell=0}^\infty(-1)^{\ell}\dim_\k \Hom_{A}([M],\Sigma^{\ell}[N]) \\
& = \langle [M], [N] \rangle_A.
\end{align*}
This implies that there is an invertible matrix $P$ such that $C_B={}^{t}PC_AP$, where ${}^tP$ is the transpose of $P$. Therefore, if $C_A$ is positive definite, then so is $C_B$. 
\end{proof}

\section{Symmetric algebras of polynomial growth}
As mentioned in the introduction, we only focus on representation-infinite symmetric algebras of polynomial growth. 
We display by Table \ref{fig1} (originally from \cite{Sk06}) the complete classifications of these algebras up to Morita and derived equivalences, in which one may refer to the corresponding reference for the definitions of relevant algebras. 
Here, we follow the convention in \cite[(1.8)]{Sk06} that an algebra $A$ is called \emph{standard} if there exists a Galois covering $\hat{A}\to \hat{A}/G=A$ such that $\hat{A}$ is a simply connected locally bounded category, and $G$ is an admissible torsion-free group of $\k$-linear automorphisms of $\hat{A}$. 
In the table, \textbf{D} stands for representation-infinite domestic and \textbf{PG} stands for non-domestic polynomial growth.
\begin{table}[htp]
\begin{center}
\renewcommand\arraystretch{1.2}
\begin{tabular}{c||c|c|c|c}
 & \fbox{Standard}  & \fbox{Cartan matrix}  & \fbox{Morita equivalence} & \fbox{Derived equivalence}\\ \hline\hline
\multirow{8}*{\textbf{D}}  &\multirow{7}*{Yes}             &\multirow{3}*{Singular} & \text{Trivial extensions of} & $T(p,q), T(2,2,r)^\ast$,  \\
                                &  &  &\text{Euclidean algebras} & $T(3,3,3), T(2,4,4)$,\\
                      &&&(\cite{BoS05})&  $T(2,3,6)$ \text{(\cite{HS})}\\ \cline{3-5}
                                                               &                        &\multirow{4}*{Non-singular} & $\Lambda(T,v_1,v_2), \Lambda'(T),$  &  \\
                                                               &                              &                                        &                              $\Gamma^{(0)}(T,v),\Gamma^{(1)}(T,v)$,& $A(p,q), \Lambda(m), \Gamma(n)$\\
                                                            &                              &                                        &   $\Gamma^{(2)}(T,v_1,v_2)$ & \text{ (\cite{HS})} \\
                                                            &&& \text{ (\cite{BoS05})}\\ \cline{2-5}

                                                            &No    & \text{Non-singular} & $\Omega(T)$ \text{(\cite{BoS06})} & $\Omega(n)$ \text{(\cite{BHS07})}\\ \hline \hline
\multirow{7}*{\textbf{PG}}&\multirow{5}*{Yes}              &\multirow{3}*{Singular} & \text{Trivial extensions of} & \text{Trivial extensions of}\\
&&&\text{tubular algebras}& \text{canonical tubular algebras}\\
&&&\text{(\cite{BS03})} & \text{(\cite{Sk06})} \\ \cline{3-5}
                                                               &                        &\multirow{2}*{Non-singular}& $\{A_i\mid i=1,2,\ldots, 16\}$& $\{A_i\mid i=1,2,3,4,5,12\}$ \\
        &&&  \text{(\cite{BS03, BHS04})} &                                       \text{(\cite{BiHS03})}            \\\cline{2-5}

                                                             &\multirow{2}*{No}   &\multirow{2}*{Non-singular} & $\{\Lambda_i\mid i=1,2,\ldots,9\}$& $\{\Lambda_i\mid i=1,3,4,9\}$\\
                                                             &         & &     \text{(\cite{BS04})} & \text{(\cite{BiHS03(2)})}\\ \cline{1-5}
\end{tabular}
\end{center}
\caption{Representation-infinite symmetric algebras of polynomial growth}\label{fig1}
\end{table}

\subsection{Reduction methods}
We need the following two crucial lemmas.

\begin{lemma}\label{sym_tau_fin}
Let $A$ be a symmetric algebra whose Cartan matrix $C_A$ is positive definite. Assume that the entries of the Cartan matrix $C_B$ are bounded for any algebra $B$ which is derived equivalent to $A$, that is, there is a positive integer $m$ such that all entries of $C_B$ are less than $m$. Then, all algebras which are derived equivalent to $A$ are $\tau$-tilting finite.
\end{lemma}
\begin{proof}
By our assumption, it follows from \cite[Theorem 13]{EJR-reduction} that $A$ is $\tau$-tilting finite.
Let $B$ be an algebra that is derived equivalent to $A$. Then, by Proposition \ref{cartan-derived}, we deduce that $C_B$ is also positive definite. We apply \cite[Theorem 13]{EJR-reduction} again and conclude that $B$ is $\tau$-tilting finite.
\end{proof}

\begin{lemma}\label{sym_triv_inf}
Let $A$ be a $\tau$-tilting infinite algebra.
Then, so is $\mathsf{Triv}(A)$.
\end{lemma}
\begin{proof}
Since there is a surjective algebra homomorphism $\mathsf{Triv}(A)\to A$, the assertion follows immediately.
\end{proof}

\subsection{Non-domestic symmetric algebras of polynomial growth}
In \cite{AHMW}, the authors have shown that any non-domestic symmetric algebra of polynomial growth whose Cartan matrix is non-singular is $\tau$-tilting finite. Here, we omit the quiver and relations of $A_i$ $(i=1,2,\ldots,16)$ and $\Lambda_i$ $(i=1,2,\ldots,9)$.
The definitions of these algebras are summarised in, for example, \cite{AHMW}.

\begin{proposition}[{\cite[Theorems 3.1 and 3.3]{AHMW}}]\label{der-fin-PG}
Let $A$ be a non-domestic symmetric algebra of polynomial growth whose Cartan matrix is non-singular.
Then, any algebra $B$ which is derived equivalent to $A$ is $\tau$-tilting finite.
\end{proposition}

In this subsection, the remaining algebras are the non-domestic symmetric algebras of polynomial growth with a singular Cartan matrix. Then, we have
\begin{proposition}\label{der-inf-PG}
Let $A$ be a non-domestic symmetric algebra of polynomial growth whose Cartan matrix is singular. Then, any algebra $B$ which is derived equivalent to $A$ is $\tau$-tilting infinite.
\end{proposition}
\begin{proof}
By Propositions \ref{sym_hiera} and \ref{cartan-derived}, $B$ must be a non-domestic symmetric algebra of polynomial growth with a singular Cartan matrix. Then $B$ is isomorphic to a trivial extension of a tubular algebra (see Table \ref{fig1}).
Since any tubular algebra admits a unique preprojective component with infinitely many vertices (for example, see \cite[XIX, Theorem 3.20]{SS07(2)}), $B$ is $\tau$-tilting infinite by Lemma \ref{sym_triv_inf}.
\end{proof}

\subsection{Representation-infinite domestic symmetric algebras}
In this subsection, we first recall the complete list of representatives for the derived equivalence classes of representation-infinite domestic symmetric algebras. Then, we determine the $\tau$-tilting finiteness of these representatives. Lastly, we use Lemmas \ref{sym_tau_fin} and \ref{sym_triv_inf} to deal with the algebras which are derived equivalent to these representatives.

Now, we define the following quivers to recall such derived equivalence classes.
\[ 
\begin{array}{cc}
\textbf{(I): \underline{$\Delta(p,q)$ with $p,q\geq 1$}}  & \textbf{(II): \underline{$\Delta(p,q,r)$ with $p,q,r\geq 1$}} \\
\scalebox{0.8}{\begin{xy}
(0,0) *[o]+{\bullet}="0", (-6,14) *[o]+{\bullet}="1",(-20,20) *[o]+{\bullet}="2", (-34,14) *[o]+{}="3",(-40,0) *[]+{}="4",  (-36,-14) *[]+{}="p-2", (-20,-20) *[o]+{\bullet}="p-1", (-6,-14) *[o]+{\bullet}="p",
(6,14) *[o]+{\bullet}="11",(20,20) *[o]+{\bullet}="12", (34,14) *[o]+{}="13",(40,0) *[]+{}="14",  (36,-14) *[]+{}="1p-2", (20,-20) *[o]+{\bullet}="1p-1", (6,-14) *[o]+{\bullet}="1p",
\ar @/^ -1.5mm/ "0";"1"^{\alpha _1}
\ar @/^ -1.5mm/ "1";"2"^{\alpha _2}
\ar @/^ -1.5mm/ "2";"3"^{\alpha _3}
\ar @{--} @/^ -5.5mm/ "3";"p-2"_{}
\ar @/^ -1.5mm/  "p-2";"p-1"^{\alpha_{p-2}}
\ar @/^ -1.5mm/  "p-1";"p"^{\alpha_{p-1}}
\ar @/^ -1.5mm/  "p";"0"^{\alpha_p}
\ar @/^ 1.5mm/ "0";"11"_{\beta _1}
\ar @/^ 1.5mm/ "11";"12"_{\beta _2}
\ar @/^ 1.5mm/ "12";"13"_{\beta _3}
\ar @{--} @/^ 5.5mm/ "13";"1p-2"_{}
\ar @/^ 1.5mm/  "1p-2";"1p-1"_{\beta_{q-2}}
\ar @/^ 1.5mm/  "1p-1";"1p"_{\beta_{q-1}}
\ar @/^ 1.5mm/  "1p";"0"_{\beta_q}
\end{xy}}
&
\scalebox{0.7}{\begin{xy}
(0,0) *[o]+{\bullet}="0", (-6,14) *[o]+{\bullet}="1",(-20,25) *[o]+{\bullet}="2", (-37,27) *[o]+{}="3",(-60,0) *[]+{}="4",  (-42,15) *[]+{}="p-2", (-30,5) *[o]+{\bullet}="p-1", (-15,0) *[o]+{\bullet}="p",
(6,14) *[o]+{\bullet}="11",(20,25) *[o]+{\bullet}="12", (37,27) *[o]+{}="13",(45,35) *[]+{}="14",  (42,15) *[]+{}="1p-2", (30,5) *[o]+{\bullet}="1p-1", (15,0) *[o]+{\bullet}="1p",
(-9,-14) *[o]+{\bullet}="21",(-11.5,-30) *[o]+{\bullet}="22", (-7,-45) *[o]+{}="23",(45,35) *[]+{}="24",  (6,-45) *[]+{}="2p-2", (11.5,-30) *[o]+{\bullet}="2p-1", (9,-14) *[o]+{\bullet}="2p",
\ar @/^ -1.5mm/ "0";"1"^{\alpha _1}
\ar @/^ -1.5mm/ "1";"2"^{\alpha _2}
\ar @/^ -1.5mm/ "2";"3"^{\alpha _3}
\ar @{--} "3";"p-2"_{}
\ar @/^ -1.5mm/  "p-2";"p-1"^{\alpha_{p-2}}
\ar @/^ -1.5mm/  "p-1";"p"^{\alpha_{p-1}}
\ar @/^ -1.5mm/  "p";"0"^{\alpha_p}

\ar @/^ 1.5mm/ "0";"11"_{\beta _1}
\ar @/^ 1.5mm/ "11";"12"_{\beta _2}
\ar @/^ 1.5mm/ "12";"13"_{\beta _3}
\ar @{--}  "13";"1p-2"_{}
\ar @/^ 1.5mm/  "1p-2";"1p-1"_{\beta_{q-2}}
\ar @/^ 1.5mm/  "1p-1";"1p"_{\beta_{q-1}}
\ar @/^ 1.5mm/  "1p";"0"_{\beta_q}

\ar @/^ -1.5mm/ "0";"21"^{\gamma _1}
\ar @/^ -1.5mm/ "21";"22"^{\gamma _2}
\ar @/^ -1.5mm/ "22";"23"^{\gamma _3}
\ar @{--}  "23";"2p-2"_{}
\ar @/^ -1.5mm/  "2p-2";"2p-1"^{\gamma_{r-2}}
\ar @/^ -1.5mm/  "2p-1";"2p"^{\gamma_{r-1}}
\ar @/^ -1.5mm/  "2p";"0"^{\gamma_r}
\end{xy}}
\end{array}
\]
\[
\begin{array}{ll}
\textbf{(III): \underline{$\Sigma(p,q)$ with $p,q\geq 1$}} & \textbf{(IV): \underline{$\Theta(r)$ with $r\geq 2$}} \\
\scalebox{0.8}{\begin{xy}
(0,5) *[o]+{\bullet}="0",
(0,-5) *[o]+{\bullet}="00",
(-8,15) *[o]+{\bullet}="1",(-20,20) *[o]+{\bullet}="2", (-34,14) *[o]+{}="3",(-40,0) *[]+{}="4",  (-36,-14) *[]+{}="p-2", (-20,-20) *[o]+{\bullet}="p-1", (-8,-15) *[o]+{\bullet}="p",
(8,15) *[o]+{\bullet}="11",(20,20) *[o]+{\bullet}="12", (34,14) *[o]+{}="13",(40,0) *[]+{}="14",  (36,-14) *[]+{}="1p-2", (20,-20) *[o]+{\bullet}="1p-1", (8,-15) *[o]+{\bullet}="1p",
\ar @/^ -1.5mm/ "0";"1"^{\alpha _1}
\ar @/^ -1.5mm/ "1";"2"^{\alpha _2}
\ar @/^ -1.5mm/ "2";"3"^{\alpha _3}
\ar @{--} @/^ -6mm/ "3";"p-2"_{}
\ar @/^ -1.5mm/  "p-2";"p-1"^{\alpha_{p-2}}
\ar @/^ -1.5mm/  "p-1";"p"^{\alpha_{p-1}}
\ar @/^ -1.5mm/  "p";"00"^{\alpha_p}

\ar @/^ 1.5mm/ "0";"11"_{\beta _1}
\ar @/^ 1.5mm/ "11";"12"_{\beta _2}
\ar @/^ 1.5mm/ "12";"13"_{\beta _3}
\ar @{--} @/^ 5.5mm/ "13";"1p-2"_{}
\ar @/^ 1.5mm/  "1p-2";"1p-1"_{\beta_{q-2}}
\ar @/^ 1.5mm/  "1p-1";"1p"_{\beta_{q-1}}
\ar @/^ 1.5mm/  "1p";"00"_{\beta_q}

\ar @<0.1cm>"00";"0"^{\gamma}
\ar @<-0.1cm>"00";"0"_{\sigma}
\end{xy}}
& \qquad\qquad
\scalebox{0.8}{\begin{xy}
(0,0) *[o]+{\bullet}="0",
(-8,15) *[o]+{\bullet}="1",(-8,-15) *[o]+{\bullet}="p",
(8,15) *[o]+{\bullet}="11",(20,20) *[o]+{\bullet}="12", (34,14) *[o]+{}="13",(40,0) *[]+{}="14",  (36,-14) *[]+{}="1p-2", (20,-20) *[o]+{\bullet}="1p-1", (8,-15) *[o]+{\bullet}="1p",
(15,8) *[o]+{\bullet}="111",
\ar @/^ 1.5mm/ "0";"11"_{\gamma _1}
\ar @/^ 1.5mm/ "11";"12"_{\gamma _2}
\ar @/^ 1.5mm/ "12";"13"_{\gamma _3}
\ar @{--} @/^ 5.5mm/ "13";"1p-2"_{}
\ar @/^ 1.5mm/  "1p-2";"1p-1"_{\gamma_{q-2}}
\ar @/^ 1.5mm/  "1p-1";"1p"_{\gamma_{q-1}}
\ar @/^ 1.5mm/  "1p";"0"_{\gamma_q}

\ar @<0.1cm>"0";"1"^{\alpha_1}
\ar @<0.1cm>"1";"0"^{\alpha_2}
\ar @<0.1cm>"0";"p"^{\beta_1}
\ar @<0.1cm>"p";"0"^{\beta_2}
\ar @/^ -1.5mm/ "0";"111"_{\sigma _1}
\ar @/^ -1.5mm/ "111";"12"_{\sigma _2}
\end{xy}}
\end{array} 
\]

\begin{definition}
We define some symmetric algebras as follows.
\begin{enumerate}
\item Set $A(p,q):=\k \Delta(p,q)/I_1(p,q)$ for any $1\leq p\leq q$, where $I_1(p,q)$ is generated by 
\begin{itemize}
\item $\alpha_1\alpha_2\cdots\alpha_p\beta_1\beta_2\cdots\beta_q= \beta_1\beta_2\cdots\beta_q\alpha_1\alpha_2\cdots\alpha_p$,
\item $\alpha_p\alpha_1=\beta_q\beta_1=0$,
\item $\alpha_{i}\alpha_{i+1}\cdots\alpha_p\beta_1\beta_2\cdots\beta_q\alpha_1\alpha_2\cdots\alpha_{i}=0$, for any $2\leq i\leq p-1$,
\item $\beta_{j}\beta_{j+1}\cdots\beta_q\alpha_1\alpha_2\cdots\alpha_{p}\beta_1\beta_2\cdots\beta_{j}=0$, for any $2\leq j\leq q-1$.
\end{itemize}

\item Set $\Lambda(m):=\k\Delta(1,m)/I_2(m)$ for any $m\geq 2$, where $I_2(m)$ is generated by
\begin{itemize}
\item $\alpha_1^2=(\beta_1\beta_2\cdots\beta_m)^2$,
\item $\alpha_1\beta_1=\beta_m\alpha_1=0$,
\item $\beta_{j}\beta_{j+1}\cdots\beta_m\beta_1\beta_2\cdots\beta_{m}\beta_1\beta_2\cdots\beta_{j}=0$, for any $2\leq j\leq m-1$.
\end{itemize}

\item Set $\Gamma(n):=\k\Delta(2,2,n)/I_3(n)$ for any $n\geq 1$, where $I_3(n)$ is generated by
\begin{itemize}
\item $\alpha_1\alpha_2 = (\gamma_1\gamma_2\cdots\gamma_n)^2=\beta_1\beta_2$,
\item $\alpha_2\gamma_1=\beta_2\gamma_1=\gamma_n\alpha_1=\gamma_n\beta_1 =\alpha_2\beta_1=\beta_2\alpha_1=0$,
\item $\gamma_{j}\gamma_{j+1}\cdots\gamma_n\gamma_1\gamma_2\cdots\gamma_{n}\gamma_1\gamma_2\cdots\gamma_{j}=0$, for any $2\leq j\leq n-1$.
\end{itemize}

\item Set $T(p,q,r):=\k\Delta(p,q,r)/I_4(p,q,r)$ for any $2\leq p\leq q \leq r$, where $I_4(p,q,r)$ is generated by
\begin{itemize}
\item $\alpha_1\alpha_2\cdots\alpha_p=\beta_1\beta_2\cdots\beta_q =\gamma_1\gamma_2\cdots\gamma_r$,
\item $\alpha_p\gamma_1=\beta_q\gamma_1=\gamma_r\alpha_1=\gamma_r\beta_1 =\alpha_p\beta_1=\beta_q\alpha_1=0$,
\item $\alpha_{i}\alpha_{i+1}\cdots\alpha_p\alpha_1\alpha_2\cdots\alpha_{i}=0$,  for any $2\leq i\leq p-1$,
\item $\beta_{j}\beta_{j+1}\cdots\beta_q\beta_1\beta_2\cdots\beta_{j}=0$, for any $2\leq j\leq q-1$,
\item $\gamma_{k}\gamma_{k+1}\cdots\gamma_r\gamma_1\gamma_2\cdots\gamma_{k}=0$,  for any $2\leq k\leq r-1$.
\end{itemize}

\item Set $T(p,q):=\k\Sigma(p,q)/I_5(p,q)$ for any $1\leq p\leq q$, where $I_5(p,q)$ is generated by
\begin{itemize}
\item $\alpha_1\alpha_2\cdots\alpha_p\gamma=\beta_1\beta_2\cdots\beta_q \sigma$,
\item $\gamma\alpha_1\alpha_2\cdots\alpha_p=\sigma\beta_1\beta_2\cdots\beta_q$,
\item $\alpha_p\sigma=\sigma\alpha_1=\beta_q\gamma=\gamma\beta_1=0$,
\item $\alpha_{i}\alpha_{i+1}\cdots\alpha_p\gamma\alpha_1\alpha_2\cdots\alpha_{i}=0$, for any $2\leq i\leq p-1$,
\item $\beta_{j}\beta_{j+1}\cdots\beta_q\sigma\beta_1\beta_2\cdots\beta_{j}=0$, for any $2\leq j\leq q-1$.
\end{itemize}

\item Set $T(2,2,r)^\ast:=\k\Theta(r)/I_6(r)$ for any $2\leq r$, where $I_6(r)$ is generated by
\begin{itemize}
\item $\alpha_1\alpha_2=\beta_1\beta_2=\gamma_1\gamma_2\cdots\gamma_r$,
\item $\gamma_r\alpha_1=\beta_2\alpha_1=\gamma_r\beta_1=\alpha_2\beta_1=\alpha_2\gamma_1=\alpha_2\sigma_1=\beta_2\gamma_1=\beta_2\sigma_1=0$,
\item $\alpha_2\alpha_1\alpha_2=\beta_2\beta_1\beta_2=0$,
\item $\gamma_2\gamma_3\cdots\gamma_r\sigma_1=\sigma_2\gamma_3\cdots\gamma_r\gamma_1=0$,
\item $\gamma_{k}\gamma_{k+1}\cdots\gamma_r\gamma_1\gamma_2\cdots\gamma_{k}=0$, for any $3\leq k\leq r-1$.
\end{itemize}

\item Set $\Omega(n):=\k\Delta(1,n)/I_7(n)$ for any $n\geq 1$, where $I_7(n)$ is generated by
\begin{itemize}
\item $\alpha_1\beta_1\beta_2\cdots\beta_n=\beta_1\beta_2\cdots\beta_n\alpha_1$,
\item $\alpha_1^2=\alpha_1\beta_1\beta_2\cdots\beta_n$,
\item $\beta_n\beta_1=0$,
\item $\beta_{k}\beta_{k+1}\cdots\beta_n\alpha_1\beta_1\beta_2\cdots\beta_{k}=0$, for any $2\leq k\leq n-1$.
\end{itemize}
\end{enumerate}
\end{definition}

We have the following results related to standard domestic symmetric algebras. Note that not all $T(p,q,r)$'s are domestic algebras.

\begin{lemma}\label{finite-A}
The following statements hold.
\begin{enumerate}
\item[\rm{(1)}] For any $1\leq p\leq q$, the algebra $A(p,q)$ is $\tau$-tilting finite.
\item[\rm{(2)}] For any $m\geq 2$, the algebra $\Lambda(m)$ is $\tau$-tilting finite.
\end{enumerate}
\end{lemma}
\begin{proof}
Observe that the algebras $A(p,q)$ and $\Lambda(m)$ are isomorphic to the Brauer graph algebras with respect to the following Brauer graphs $G_1$ and $G_2$, respectively.
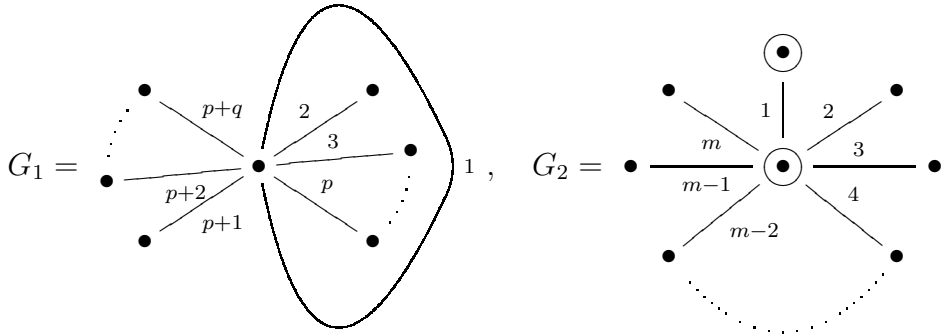
\[
G_1=
\scalebox{0.65}{\begin{xy}
(0,0) *[o]+{\bullet}="0", (25,3) *{}="dm1",(25,-3) *{}="dm2",
(15,10) *[o]+{\bullet}="2", (20,2) *[o]+{\bullet}="3", (15,-10) *[o]+{\bullet}="p",
(-15,10) *[o]+{\bullet}="q", (-20,-2) *[o]+{\bullet}="33", (-15,-10) *[o]+{\bullet}="22"
\ar @{-}  "0";"2"^{2}
\ar @{-}  "0";"3"^{3}
\ar @/^ 1.5mm/  @{..}"3";"p"_{}
\ar @{-}  "0";"p"^{p}
\ar @{-}  "0";"22"^{p+1}
\ar @{-}  "0";"33"^{p+2}
\ar @/^ 1.5mm/  @{..}"33";"q"_{}
\ar @{-}  "0";"q"_{p+q}
\ar @(u,u)  @/^ 20mm/   @{-}"0";"dm1"
\ar @(u,u)  @/^ 0.5mm/   @{-}"dm1";"dm2"^{1}
\ar @(u,u)  @/^ 20mm/   @{-}"dm2";"0"
\end{xy}},
\qquad 
G_2=\scalebox{0.65}{
\begin{xy}
(0,0) *+[o][F-]+{\bullet}="0", (0,15) *+[o][F-]+{\bullet}="1",(25,-3) *{}="dm2",
(15,10) *[o]+{\bullet}="2", (20,0) *[o]+{\bullet}="3", (15,-12) *[o]+{\bullet}="4",
(-15,10) *[o]+{\bullet}="q", (-20,0) *[o]+{\bullet}="33", (-15,-12) *[o]+{\bullet}="22",
\ar @{-}  "0";"1"^{1}
\ar @{-}  "0";"2"^{2}
\ar @{-}  "0";"3"^{3}
\ar @{-}  "0";"4"^{4}
\ar @{-}  "0";"33"^{m-1}
\ar @{-}  "0";"22"^{m-2}
\ar @/^ 10mm/  @{..}"4";"22"_{}
\ar @{-}  "0";"q"^{m}
\end{xy}}
 \]
Here, the circled vertices have multiplicity 2.
We then apply \cite[Theorem 6.7]{AAC} to deduce that $A(p,q)$ and $\Lambda(m)$ are $\tau$-tilting finite.
\end{proof}

\begin{lemma}\label{finite-G}
For any $n\geq 1$, the algebra $\Gamma(n)$ is $\tau$-tilting finite.
\end{lemma}
\begin{proof}
The Cartan matrix $C_{\Gamma(n)}$ is a square matrix of size $n+2$ of the form
\begin{center}
\scalebox{0.9}{$\begin{pmatrix}
2 & 0 & 1 & 0 & \cdots & 0 & 0 \\
0 & 2 & 1 & 0 & \cdots & 0 & 0 \\
1 & 1 & 3 & 2 & \cdots & 2 & 2 \\
0 & 0 & 2 & 3 & \cdots & 2 & 2 \\
\vdots & \vdots & \vdots & \vdots & \cdots & \vdots & \vdots \\
0 & 0 & 2 & 2 & \cdots & 3 & 2 \\
0 & 0 & 2 & 2 & \cdots & 2 & 3 \\
\end{pmatrix}.$}
\end{center}
It is easy to check that $C_{\Gamma(n)}$ is positive definite and $\mathsf{det}(C_{\Gamma(n)})=8$. Let $B$ be an algebra derived equivalent to $\Gamma(n)$. By \cite[Theorem 3.1]{BoS05} (see also \cite[Theorem 1]{BoS03}), one knows that $B$ is isomorphic to one of $\Gamma^{(0)}(T,v)$, $\Gamma^{(1)}(T,v)$ and $\Gamma^{(2)}(T,v_1,v_2)$ (see \cite{BoS05} for the definitions).
The algebras $\Gamma^{(0)}(T,v)$, $\Gamma^{(1)}(T,v)$ and $\Gamma^{(2)}(T,v_1,v_2)$ are determined by a Brauer graph $T$ whose vertex has at most multiplicity $2$. Since there are only finitely many such Brauer graphs with $n+2$ edges, we have only finitely many choices of $\Gamma^{(0)}(T,v)$, $\Gamma^{(1)}(T,v)$ and $\Gamma^{(2)}(T,v_1,v_2)$ having $n+2$ simple modules.
Thus, the entries of the Cartan matrix $C_B$ are bounded. Then, we deduce the statement by Lemma \ref{sym_tau_fin}.
\end{proof}

\begin{lemma}\label{infinite-T}
The following statements hold.
\begin{enumerate}
\item[\rm{(1)}] For any $1\leq p\leq q$, the algebra $T(p,q)$ is $\tau$-tilting infinite.
\item[\rm{(2)}] For any $r\geq 2$, the algebra $T(2,2,r)^{\ast}$ is $\tau$-tilting infinite.
\item[\rm{(3)}] The algebras $T(3,3,3)$, $T(3,4,4)$, $T(2,3,6)$ are $\tau$-tilting infinite.
\end{enumerate}
\end{lemma}
\begin{proof}
(1) We notice that there is a surjection $T(p,q)\to \k (\xymatrix@C=0.7cm{\bullet \ar@<0.5ex>[r] \ar@<-0.5ex>[r] &\bullet})$. Since the Kronecker algebra $\k (\xymatrix@C=0.7cm{\bullet \ar@<0.5ex>[r] \ar@<-0.5ex>[r] &\bullet})$ is $\tau$-tilting infinite, so is $T(p,q)$.

(2) For any $r\geq 2$, there is a surjection $T(2,2,r)^{\ast}\to \k \widetilde{\mathbb{D}}_4$. Thus, $T(2,2,r)^{\ast}$ is $\tau$-tilting infinite.

(3) We set $A_1:=T(3,3,3)$, $A_2:=T(3,4,4)$ and $A_3:=T(2,3,6)$.
Then there is a surjection from $A_i$ to $\k\Delta_i$, where $\Delta_1=\widetilde{\mathbb{E}}_6$, $\Delta_2=\widetilde{\mathbb{E}}_7$ and $\Delta_3=\widetilde{\mathbb{E}}_8$. Hence, we conclude that $A_i$ is $\tau$-tilting infinite for $i=1,2,3$.
\end{proof}

Note that the non-standard domestic symmetric algebras occur only when the characteristic of $\k$ is $2$. In this case, a symmetric algebra $A$ is non-standard domestic if and only if $A$ is derived equivalent to $\Omega(n)$, for some $n\geq 1$ (\cite{BHS07}). Moreover, the algebras $\Omega(n)$'s form a complete set of representatives of pairwise different derived equivalence classes of representation-infinite non-standard domestic symmetric algebras (\cite[Theorem 2.2]{HS}).
\begin{lemma}\label{finite-omega}
For any $n\geq 1$, the algebra $\Omega(n)$ is $\tau$-tilting finite.
\end{lemma}
\begin{proof}
We notice that $\Omega(n)$ and $A(1,n)$ are socle equivalent, that is, $\Omega(n)/\soc~\Omega(n)$ and $A(1,n)/\soc A(1,n)$ are isomorphic.
It follows from \cite[Theorem 3.3 (2)]{Ada1} that there are poset isomorphisms
\[ 
\stautilt \Omega (n) \simeq \stautilt  (\Omega(n)/\soc~\Omega(n)) \simeq \stautilt  (A(1,n)/\soc~A(1,n)) \simeq \stautilt  A(1,n). 
\]
By Lemma \ref{finite-A}, $\Omega(n)$ is $\tau$-tilting finite.
\end{proof}

We recall from Table \ref{fig1} that the complete list of derived equivalence classes of representation-infinite domestic symmetric algebras is divided into the following cases:
\begin{itemize}
  \item standard representation-infinite domestic symmetric algebras with singular Cartan matrix, i.e., $T(p,q)$, $T(2,2,r)^\ast$, $T(3,3,3)$, $T(2,4,4)$, $T(2,3,6)$;
  \item standard representation-infinite domestic symmetric algebras with non-singular Cartan matrix, i.e., $A(p,q)$, $\Lambda(m)$, $\Gamma(n)$;
  \item non-standard representation-infinite domestic symmetric algebras with non-singular Cartan matrix, i.e., $\Omega(n)$.
\end{itemize}

\begin{proposition}\label{der-inf-D}
Any algebra $A$ which is derived equivalent to one of algebras $T(p,q)$, $T(2,2,r)^{\ast}$, $T(3,3,3)$, $T(2,4,4)$ and $T(2,3,6)$ is $\tau$-tilting infinite.
\end{proposition}
\begin{proof}
Let $A$ be a symmetric algebra which is derived equivalent to one of the algebras $T(p,q)$, $T(2,2,r)^{\ast}$, $T(3,3,3)$, $T(2,4,4)$ and $T(2,3,6)$. Then, it follows from \cite[Theorem 2]{BoS05} that the algebra $A$ is isomorphic to a trivial extension $\mathsf{Triv}(B)$, where $B$ is a representation-infinite tilted algebra of Euclidean type.
Since $B$ is $\tau$-tilting infinite (see \cite{Z}), the assertion follows from Lemma \ref{sym_triv_inf}.
\end{proof}

\begin{proposition}\label{der-fin-D}
Any algebra $B$ which is derived equivalent to one of algebras $A(p,q)$, $\Lambda(m)$, $\Gamma(n)$ and $\Omega(n)$ is $\tau$-tilting finite.
\end{proposition}
\begin{proof}
Let $A\in \{A(p,q), \Lambda(m), \Gamma(n), \Omega(n)\}$.
Then, it follows from \cite[Theorems 1.1 and 1.2]{BHS04} that $A$ has a non-singular Cartan matrix.
Moreover, the entries of the Cartan matrices of algebras which are derived equivalent to $A$ are bounded.
(Same arguments in the proof of Lemma \ref{finite-G} work for $A(p,q)$, $\Lambda(m)$, and $\Omega(n)$.)
Then the assertion follows from Lemmas \ref{finite-A}, \ref{finite-G}, \ref{finite-omega} and Proposition \ref{cartan-derived}.
\end{proof}

Summing up the above statements,  we can give the main result of this section.
\begin{theorem}\label{tau-finite_PG}
Let $A$ and $B$ be two representation-infinite symmetric algebras of polynomial growth. If $A$ is derived equivalent to $B$, then the following conditions are equivalent.
\begin{enumerate}
\item[\rm{(1)}] $A$ is $\tau$-tilting finite.
\item[\rm{(2)}] $B$ is $\tau$-tilting finite.
\item[\rm{(3)}] The Cartan matrix $C_A$ (or equivalently, $C_B$) is non-singular.
\end{enumerate}
\end{theorem}
\begin{proof}
We divide representation-infinite symmetric algebras of polynomial growth into the following subclasses,
\begin{itemize}
  \item representation-infinite domestic symmetric algebras with singular Cartan matrix;
  \item representation-infinite domestic symmetric algebras with non-singular Cartan matrix;
  \item non-domestic symmetric algebras of polynomial growth with singular Cartan matrix;
  \item non-domestic symmetric algebras of polynomial growth with non-singular Cartan matrix.
\end{itemize}
Then the statement follows from Propositions  \ref{der-fin-PG}, \ref{der-inf-PG},  \ref{der-inf-D}, and \ref{der-fin-D}
\end{proof}

A self-injective algebra $A$ is said to be \emph{tilting-discrete} if, for any $n\geq 1$, there are only finitely many isomorphism classes of basic $n$-term tilting complexes of $\Kb(\proj A)$.

\begin{corollary}
A symmetric algebra of polynomial growth whose Cartan matrix is non-singular is tilting-discrete.
\end{corollary}
\begin{proof}
The statement follows from Theorem \ref{tau-finite_PG} and \cite[Corollary 2.11]{AM}.
\end{proof}

As another application of Theorem \ref{tau-finite_PG}, we find that any self-injective cellular algebra of polynomial growth is $\tau$-tilting finite.

\begin{corollary}\label{finite_cellular}
Suppose the characteristic of $\k$ is not $2$. Then, any (not necessarily basic) algebra which is Morita equivalent to a self-injective cellular algebra of polynomial growth is $\tau$-tilting finite.
\end{corollary}
\begin{proof}
Since the characteristic of the base field $\k$ is not $2$, cellularity is preserved under Morita equivalence so that we may assume that algebras are basic.

Any basic indecomposable self-injective cellular algebra of polynomial growth is symmetric by \cite[Theorem 1]{AKMW}. Since any cellular algebra has a non-singular Cartan matrix \cite[Theorem 3.7]{GL}, the assertion follows from Theorem \ref{tau-finite_PG}.
\end{proof}

\section{\texorpdfstring{$0$}{0}-Hecke algebras and \texorpdfstring{$0$}{0}-Schur algebras}\label{section-4}
In this section, we focus on 0-Hecke algebras and 0-Schur algebras. We mention that these algebras are not necessarily connected and symmetric. In particular, a characterization of $0$-Hecke algebras being symmetric is given in \cite[Proposition 4.5]{Fayers-0-hecke-alg}.

\subsection{\texorpdfstring{$0$}{0}-Hecke algebras}
We recall some background of Coxeter groups and Hecke algebras. We refer to the textbook \cite{BB-coxeter-group} for more details.

Let $W$ be a Coxeter group with the generating set $S=\{s_1,s_2,\ldots,s_n\}$, that is, the group with a presentation of the form
\[
W=\langle s_1,s_2,\ldots ,s_n \mid (s_is_j)^{m_{ij}}=1 \rangle, 
\]
where $m_{ii}=1$ and $m_{ij}\in \{2,3,4,\ldots\}$ for $i\neq j$. The pair $(W,S)$ is called a \emph{Coxeter system}. 
We assign to a Coxeter system $(W,S)$ a Coxeter diagram $\Gamma(W,S)$ as follows: the vertex set of $\Gamma(W,S)$ is in bijection with $S$; two vertices $s_i$ and $s_j$ of $\Gamma(W,S)$ are joined by an edge whenever $m_{ij}\geq 3$, then this edge is labeled by $m_{ij}$ and we omit the label on the edge $\xymatrix@C=0.7cm{s_i\ar@{-}[r]&s_j}$ if $m_{ij}=3$. 
A Coxeter system $(W,S)$ is said to be \emph{irreducible} if its Coxeter diagram $\Gamma(W,S)$ is connected. 
A Coxeter system $(W,S)$ is said to be \emph{finite} if $\Gamma(W,S)$ is a disjoint union of types $\mathbb{A}_n~(n\geq 1)$, $\mathbb{B}_n=\mathbb{C}_n ~(n\geq 2)$, $\mathbb{D}_n~(n\geq 4)$, $\mathbb{E}_6$, $\mathbb{E}_7$, $\mathbb{E}_8$, $\mathbb{F}_4$, $\mathbb{H}_3$, $\mathbb{H}_4$ and $\mathbb{I}_2(m)$ $(m\geq 5)$; this gives a complete classification of finite Coxeter groups, see \cite{coxeter}. 

Let $(W, S)$ be a finite Coxeter system and $q$ an indeterminate element.
The \emph{Iwahori-Hecke algebra} (or briefly, Hecke algebra) $\H_{\k,q}(W)$ is the $\k$-algebra generated by $\{T_{s} \mid s\in S\}$ with the following two kinds of relations:
\begin{itemize}
\item Quadratic relations: $T_{s}^2=(q-1)T_{s}+q$, for any $s\in S$.
\item Braid relations: $\underset{m_{st}\ \text{factors}}{\underbrace{T_sT_tT_sT_t\cdots}}=\underset{m_{ts}\ \text{factors}}{\underbrace{T_tT_sT_tT_s\cdots}}$, for any $s,t \in S$.
\end{itemize}
It is well-known that if $w=s_1s_2\cdots s_r$ is a reduced expression of $w\in W$, then $T_w=T_{s_1}T_{s_2}\cdots T_{s_r}$ is well-defined. The Hecke algebra $\H_{\k,q}(W)$ is a $q$-deformation of the group algebra $\k W$. Indeed, $\H_{\k,1}(W)\simeq \k W$ if $q=1$. Another natural specialization of $\H_{\k,q}(W)$ is the case $q=0$. We call $\H_0(W):=\H_{\k,0}(W)$ the \emph{$0$-Hecke algebra}.

Our aim in this subsection is to classify $\H_0(W)$ by $\tau$-tilting finiteness. To do this, it suffices to check the quiver of $\H_0(W)$, say $Q_W$, as we shall explain below.

Let $(W,S)$ be a finite Coxeter system with $S=\{s_i\ |\ i\in I\}$ and $I=\{1,2,\ldots,n \}$.
According to \cite[Theorem 4.21]{Norton-0-hecke-alg} and \cite[Theorem 5.1]{Fayers-0-hecke-alg}, the quiver $Q_W$ of $\H_0(W)$ is obtained as follows:
\begin{itemize}
\item The vertex set is $\{ v_J\mid J\subseteq I\}$.
\item We draw an arrow $v_J\to v_{K}$ whenever neither $J$ nor $K$ is contained in the other, and $m_{jk}\geq 3$, for any $j\in J\setminus K$ and $k\in K\setminus J$.
\end{itemize}
Obviously, $v_{\varnothing}$ and $v_I$ are two isolated vertices in $Q_W$, that is, there is no arrow starting or ending at $v_{\varnothing}$ and $v_I$.

\begin{example}\label{example-0-Hecke}
Let $W=\S_3$ be the symmetric group of degree 3 and $S=\{s_1,s_2,s_3\}$. Then, the Coxeter diagram of $(\S_3,S)$ is displayed as
\[ 
\xymatrix@C=0.8cm{s_1\ar@{-}[r]&s_2\ar@{-}[r]&s_3}
\]
and $(m_{ij})_{i,j\in \{1,2,3\}}$ is given by
\[ 
\begin{pmatrix}
2 &3&2\\
3&2&3\\
2&3&2
\end{pmatrix}.
\]
Therefore, the quiver $Q_W$ consists of the following three components.
\[
\vcenter{\xymatrix@C=1cm@R=0.7cm{v_\varnothing &v_{\{1\}}\ar@<0.5ex>[r]^{\ } &v_{\{2\}}\ar@<0.5ex>[l]^{\ }\ar@<0.5ex>[r]^{\ } \ar@<0.5ex>[d]^{\ }   &v_{\{3\}}\ar@<0.5ex>[l]^{\ } & v_{\{1,2,3\}}\\
&v_{\{1,2\}}\ar@<0.5ex>[r]^{\ } &v_{\{1,3\}}\ar@<0.5ex>[r]^{\ } \ar@<0.5ex>[l]^{\ } \ar@<0.5ex>[u]^{\ } &v_{\{2,3\}}\ar@<0.5ex>[l]^{\ } &}}\]
\end{example}

The following observations effectively show the $\tau$-tilting finiteness of algebras.
Indeed, the following lemmas allow us to prove that a bound quiver algebra $A\simeq \k Q/\I$ is $\tau$-tilting infinite only by the shape of $Q$.
\begin{lemma}[{\cite[Lemma 2.15]{W2}}]\label{tau-tilt-infinite-quiver}
Any bound quiver algebra $A=\k Q/\I$ is $\tau$-tilting infinite if $Q$ is $\Delta_1$ or $\Delta_2$ in the following. 
\begin{center}
$\Delta_1=\vcenter{\xymatrix@C=1cm@R=0.7cm{ \circ \ar@<0.5ex>[d]^{}\ar@<0.5ex>[r]^{}&\circ  \ar@<0.5ex>[l]^{}\ar@<0.5ex>[d]^{}\\ \circ  \ar@<0.5ex>[u]^{}\ar@<0.5ex>[r]^{ }&\circ  \ar@<0.5ex>[u]^{ }\ar@<0.5ex>[l]^{}}}\qquad $ 
$\Delta_2=\vcenter{\xymatrix@C=1cm@R=0.7cm{ \circ \ar@<0.5ex>[r]^{}&\circ  \ar@<0.5ex>[l]^{}\ar@<0.5ex>[r]^{}\ar@<0.5ex>[d]^{}&\circ\ar@<0.5ex>[l]^{}\\ \circ   \ar@<0.5ex>[r]^{ }&\circ  \ar@<0.5ex>[u]^{ }\ar@<0.5ex>[r]^{}\ar@<0.5ex>[l]^{}&\circ\ar@<0.5ex>[l]^{}}}$ 
\end{center}
\end{lemma}

\begin{lemma}[{\cite[Proposition 3.1]{MW}}]\label{tau-tilting-infinite-tensor-product}
Let $A:=\k (\xymatrix{\circ\ar@<0.5ex>[r]^{\alpha} & \circ \ar@<0.5ex>[l]^{\beta}})/\langle\alpha\beta,\beta\alpha\rangle$. Then, the tensor product algebra $A\otimes_{\k } A$  is $\tau$-tilting infinite.
\end{lemma}

We now formulate our main result in this subsection as follows.
\begin{theorem}\label{result-0-Hecke}
Let $W$ be an irreducible finite Coxeter group. Then, the following statements are equivalent.
\begin{enumerate}
\item[\rm{(1)}] The $0$-Hecke algebra $\H_0(W)$ is $\tau$-tilting finite.
\item[\rm{(2)}] $W$ is of type $\mathbb{A}_1$, $\mathbb{A}_2$, $\mathbb{B}_2$ or $\mathbb{I}_2(m)$ for $m\geq 5$.
\end{enumerate}
\end{theorem}
\begin{proof}
We first show that (2) implies (1). $\H_0(\S_1)$ is $\tau$-tilting finite since it is semisimple.
Suppose that $W$ is of type $\mathbb{A}_2$, $\mathbb{B}_2$ or $\mathbb{I}_2(m)$. Then the quiver $Q_W$ is given by
\[ 
\xymatrix@C=1cm@R=0.7cm{v_\varnothing &v_{\{1\}}\ar@<0.5ex>[r]^{\ } &v_{\{2\}}\ar@<0.5ex>[l]^{\ } & v_{\{1,2\}}}.
\]
It is known from \cite{BG} that any finite-dimensional algebra with quiver $\xymatrix@C=0.7cm{\circ\ar@<0.5ex>[r]^{\ } &\circ\ar@<0.5ex>[l]^{\ }}$ is representation-finite. We deduce that $\H_0(W)$ is representation-finite.

Suppose that $W$ is one of types $\mathbb{A}_n (n\geq 3)$, $\mathbb{B}_n=\mathbb{C}_n (n\geq 3)$, $\mathbb{D}_n (n\geq 4)$, $\mathbb{E}_6$, $\mathbb{E}_7$, $\mathbb{E}_8$, $\mathbb{F}_4$, $\mathbb{H}_3$ or $\mathbb{H}_4$.
In this case, the quiver $Q_W$ admits $\Delta_2$, which is described in Lemma \ref{tau-tilt-infinite-quiver}, as a full subquiver since the Coxeter diagram contains $\mathbb{A}_3$ as a subgraph.
We conclude that $\H_0(W)$ is $\tau$-tilting infinite.
\end{proof}

We are able to determine the $\tau$-tilting finiteness of $\H_0(W)$ for an arbitrary finite Coxeter group $W$.
If $W$ is a product of irreducible Coxeter groups $W_1, W_2, \ldots , W_n$, then $\H_0(W)$ is isomorphic to the tensor product $\H_0(W_1)\otimes \H_0(W_2)\otimes\cdots \otimes \H_0(W_n)$.
By combining Lemma \ref{tau-tilting-infinite-tensor-product} and Theorem \ref{result-0-Hecke}, we have the following result.

\begin{corollary}
Assume that $W=W_1\times W_2\times \cdots \times W_n$ is a finite Coxeter group, where $n\geq 2$ and $W_1, W_2, \ldots , W_n$ are irreducible. Then, $\H_0(W)$ is $\tau$-tilting finite if and only if one of the following conditions holds.
\begin{enumerate}
\item[\rm{(1)}] $W_1, W_2, \ldots, W_n$ are of type $\mathbb{A}_1$.
\item[\rm{(2)}] There is precisely one $W_i$ such that it is one of types $\mathbb{A}_2$, $\mathbb{B}_2$ or $\mathbb{I}_2(m)$, and the others are of type $\mathbb{A}_1$.
\end{enumerate}
\end{corollary}

Thanks to the results given by Deng and Yang in \cite{DY-rep-type-0-Hecke-alg}, we can give a new class of algebras in which $\tau$-tilting finiteness and representation-finiteness are equivalent.

\begin{corollary}\label{repfin=taufin1}
Let $W$ be an arbitrary finite Coxeter group. Then, the $0$-Hecke algebra $\H_0(W)$ is $\tau$-tilting finite if and only if it is representation-finite.
\end{corollary}

\subsection{\texorpdfstring{$0$}{0}-Schur algebras}
In this subsection, we consider the finite Coxeter group of type $\mathbb{A}_{r-1}$, i.e., the symmetric group $\S_r$.
We write $\H_q(\S_r)$ for the Hecke algebra $\H_{\k,q}(\S_r)$.
We recall the definition of $q$-Schur algebras introduced by Dipper and James \cite{DJ-q-schur-alg} as follows.

Let $r$ be a positive integer.
A \emph{composition} of $r$ is a sequence $\lambda=(\lambda_1,\lambda_2,\ldots)$ of non-negative integers such that $\sum_{i\in\mathbb{N}}^{}\lambda_i=r$.
The entries $\lambda_i$, for $i\geq 1$, are called \emph{parts} of $\lambda$.
If $\lambda_i=0$ for all $i>n$, we identify $\lambda$ with $(\lambda_1,\lambda_2,\ldots,\lambda_n)$, and we call it a composition of $r$ with at most $n$ parts. We denote by $\Omega(r)$ the set of all compositions of $r$ and by $\Omega(n,r)$ the set of all compositions of $r$ with at most $n$ parts.

For $\lambda=(\lambda_1,\lambda_2,\ldots,\lambda_n)\in \Omega(n, r)$, the corresponding \emph{Young subgroup} $\S_\lambda$ of the symmetric group $\S_r$ is defined as $\S_\lambda= \S_{\lambda_1}\times \S_{\lambda_2}\times \cdots \times \S_{\lambda_{n}}$.
Then, the \emph{$q$-Schur algebra} associated with $\H_q(\S_r)$ is the endomorphism algebra
\[ S_q(n,r):=\End_{\H_q(\S_r)}\left ( \underset{\lambda\in \Omega(n,r)}{\bigoplus} x_\lambda \H_q(\S_r) \right ), \]
where $x_\lambda=\sum_{w\in \S_\lambda}^{}T_w$. We call the degeneration at $q=0$ of $S_q(n,r)$ the \emph{$0$-Schur algebra}, and denote it by $S_0(n,r)$. When $q=1$, the algebra $S_{1}(n,r)$ is the so-called classical Schur algebra. 

For $J\subset \{1,2,\ldots, r-1\}$, let $e_J$ be the idempotent of $\H_0(\S_r)$ associated with the vertex $v_J\in Q_{\S_r}$. Then, $\{e_J\mid J\subseteq \{1,2,\ldots,r-1\}\}$ is a complete set of primitive orthogonal idempotents of $\H_0(\S_r)$ as we mentioned in the previous subsection. Let $|J|$ be the number of elements in $J$.

\begin{lemma}[{\cite[(4.3.3)]{DY-0-Schur-alg}}]\label{0-schur-alg-reduction}
For any $n\geq 2$ and $r\geq 2$, the $0$-Schur algebra $S_0(n,r)$ is Morita equivalent to $e[n]\H_0(\S_r)e[n]$, where
\[ 
e[n]=\ssum_{|J|\leq n-1}^{}e_J. 
\]
\end{lemma}

Now, we are able to classify $\tau$-tilting finite $0$-Schur algebras as follows.
\begin{theorem}\label{main-0-schur}
The following assertions hold.
\begin{enumerate}
\item[\rm{(1)}] For any $n\geq 3$, the $0$-Schur algebra $S_0(n,r)$ is $\tau$-tilting finite if and only if $r=2,3$.
\item[\rm{(2)}] For $n =2$ and $r\geq 2$, the $0$-Schur algebra $S_0(2,r)$ is $\tau$-tilting finite.
\end{enumerate}
\end{theorem}
\begin{proof} 
(1) We show the necessity.
When $r\leq n$, $S_0(n,r)$ is $\tau$-tilting finite if and only if $\H_0(\S_r)$ is $\tau$-tilting finite by Lemma \ref{0-schur-alg-reduction}, and so is $S_0(n,r)$ with $r\leq 3$.

We now show the sufficiency.
Suppose $r\geq 4$.
In this case, the quiver of $e[n]\H_0(\S_r)e[n]$ admits the following full subquiver:
\[ 
\xymatrix@C=1cm{ v_{\{1,2\}} \ar@<0.5ex>[r]^{} &v_{\{1,3\}}  \ar@<0.5ex>[d]^{}\ar@<0.5ex>[r]^{}\ar@<0.5ex>[l]^{}&v_{\{2,3\}}   \ar@<0.5ex>[l]^{}\\ v_{\{1\}} \ar@<0.5ex>[r]^{} &v_{\{2\}}   \ar@<0.5ex>[u]^{}\ar@<0.5ex>[l]^{}\ar@<0.5ex>[r]^{ }&v_{\{3\}}  \ar@<0.5ex>[l]^{}} 
\]
Thus, $S_0(n,r)$ is $\tau$-tilting infinite by Lemma \ref{tau-tilt-infinite-quiver}.

(2) It is known from \cite[Section 5]{DY-0-Schur-alg} that, for any $r\geq 2$, the 0-Schur algebra $S_0(2,r)$ is Morita equivalent to the preprojective algebra of Dynkin type $\mathbb{A}_{r-1}$, which is $\tau$-tilting finite following \cite[Theorem 2.21]{Miz1}.
\end{proof}

We mention that the representation type of $S_0(n,r)$ is determined in \cite{DY-0-Schur-alg}. Then, the relation between representation finiteness and $\tau$-tilting finiteness of $0$-Schur algebras is displayed as follows.

\begin{corollary}\label{repfin=taufin2}
Let $n\geq 3$. Then, the $0$-Schur algebra $S_0(n,r)$ is $\tau$-tilting finite if and only if it is representation-finite.
\end{corollary}
\begin{proof}
Combine Theorem \ref{main-0-schur} and \cite[Theorem 4.6]{DY-0-Schur-alg}.
\end{proof}

\section*{Acknowledgment}
The authors would like to express their gratitude to Professor Susumu Ariki for introducing them to this area of research and for his valuable comments and suggestions. Additionally, the authors are deeply thankful to Ryoichi Kase for his insightful suggestions, this paper would not have been complete without his comments. KM was partly supported by JSPS Grant-in-Aid for Early-Career Scientists (Grant No. 20K14302 and 24K16885), Grant-in-Aid for Scientific Research (A) (Grant No. 23H00479) and FY 2023 Research Project Expense Subsidy Program: Research Network Formation Project (National Institute of Technology, Japan). QW is partially supported by JSPS Grant-in-Aid for JSPS Fellows (Grant No. 20J10492), National Key Research and Development Program of China (Grant No. 2020YFA0713000) and China Postdoctoral Science Foundation (Grant No. YJ20220119 and No. 2023M731988).





\end{document}